\documentclass[11pt]{article}
\usepackage{amsfonts,amsmath,amssymb,amsthm}
\usepackage{url}
\usepackage{hyperref}
\usepackage{epsfig,latexsym,graphicx}
\usepackage{esint}
\usepackage{graphicx,latexsym,amssymb,amsmath,amsfonts,fancyhdr,mathtools}
\usepackage{euscript}
\usepackage[letterpaper,hmargin=1.0in,vmargin=1.0in]{geometry}

\newtheorem{maintheorem}{Theorem}

\newtheorem{theorem}{Theorem}[section]
\newtheorem{lemma}[theorem]{Lemma}

\newtheorem{proposition}[theorem]{Proposition}

\newtheorem{remark}{Remark}

\def\XXint#1#2#3{{\setbox0=\hbox{$#1{#2#3}{\int}$ }
\vcenter{\hbox{$#2#3$ }}\kern-.6\wd0}}

\newcommand{\defeq}{\vcentcolon=}

\newcommand{\E}{\mathbb{E}}
\renewcommand{\P}{\mathbb{P}}
\newcommand{\Var}{\mathrm{Var}}

\newcommand{\cd}{\mathcal{D}}

\newcommand{\cg}{\mathcal{G}}

\newcommand{\N}{\mathbb{N}}

\newcommand{\cf}{\mathcal{F}}

\begin{document}
\title{Limit Theorems for Longest Monotone Subsequences in Random Mallows Permutations}

\author{Riddhipratim Basu
\thanks{Dept. of Mathematics, Stanford University. Email: rbasu@stanford.edu}
\and
Nayantara Bhatnagar
\thanks{Dept. of Mathematical Sciences, University of Delaware. Email: naya@math.udel.edu}
}
\date{}
\maketitle
\begin{abstract}
We study the lengths of monotone subsequences for permutations drawn from the Mallows measure. The Mallows measure was introduced by Mallows in connection with ranking problems in statistics. Under this measure, the probability of a permutation $\pi$ is proportional to $q^{{\rm inv}(\pi)}$ where $q$ is a positive parameter and ${\rm inv}(\pi)$ is the number of inversions in $\pi$.

In our main result we show that when $0<q<1$, then the limiting distribution of the longest increasing subsequence (LIS) is Gaussian, answering an open question in \cite{BhaPel15}. This is in contrast to the case when $q=1$ where the limiting distribution of the LIS when scaled appropriately is the GUE Tracy-Widom distribution. We also obtain a law of large numbers for the length of the longest decreasing subsequence (LDS) and identify the limiting constant, answering a further open question in \cite{BhaPel15}.
\end{abstract}

\section{Introduction}
\label{s:intro}
Random permutations are well-studied objects in combinatorics and probability. Whereas different statistics of a uniform random permutation have been extensively studied, some non-uniform measures on permutations have, in recent years, generated attention as well. Among the non-uniform models of permutations, the following exponential family of distributions on $\mathcal S_n$, the set of permutations on $[n]:=\{1,2,\ldots ,n\}$, introduced by Mallows \cite{Mal57}, has been one of the popular choices. Let $q$ be a positive parameter. A random permutation $\Pi=\Pi_{n,q}$ on $\mathcal{S}_{n}$ is said to be drawn from the ${\rm Mallows}(q)$ measure if
\[
\P(\Pi = \pi) = \frac{q^{{\rm inv(\pi)}}}{Z_{n,q}}
\]
for every permutation $\pi$ of $[n]$ where
\[
{\rm inv}(\pi) \defeq \# \{ (i,j) \ : \ 1 \le i < j \le n, \pi(i) > \pi(j)       \}
\]
is the number of inversions in $\pi$ and $Z_{n,q}$ is a normalizing constant. Observe that for $q=1$, this reduces to the uniform measure on permutations whereas for $q>1$ and $q<1$, permutations with more and less inversions are favored respectively. 

In this paper, we study the lengths of longest monotone subsequence of a ${\rm Mallows}(q)$ permutation. This is a classically studied question for uniform permutations and the study for general Mallows measures was initiated following a question raised by Borodin et al. in \cite{BorDiaFul10} and has received significant attention of late \cite{MueSta11, BhaPel15}; see Section \ref{s:background} for more details. For a permutation $\pi$ in $\mathcal S_n$, we say that $1\leq i_1< i_2<\cdots < i_k$ is an {\bf increasing subsequence} of $\pi$ with {\bf length} $k$ if $\pi(i_1)<\pi(i_2)<\cdots <\pi(i_k)$. An increasing subsequence of maximum length is called a {\bf longest increasing subsequence}(LIS). Define the length of any such subsequence as the {\bf length of a longest increasing subsequence} of $\pi$. Analogously, we say that $1\leq i_1< i_2<\cdots < i_k$ is a {\bf decreasing subsequence} of $\pi$ with {\bf length} $k$ if $\pi(i_1) > \pi(i_2) > \cdots > \pi(i_k)$. A decreasing subsequence of maximum length is called a {\bf longest decreasing subsequence}(LDS) and the common length of such subsequences is defined to be the {\textbf{length of a longest decreasing subsequence}} of $\pi$. For a random $\mbox{Mallows}(q)$ permutation $\Pi$ as defined above, let $L_n=L_n(q)$ and $L^{\downarrow}_n=L^{\downarrow}_n(q)$ denote the length of a longest increasing subsequence and the length of a longest decreasing subsequence of $\Pi$ respectively. 

Weak laws of large numbers have been established for $L_n(q)$ for different ranges of $q=q(n)$. For $q=1$, it is a classical result that $L_n$ scales as $2\sqrt{n}$. It is also not hard to see via a subadditive argument (see e.g.\ \cite{BhaPel15}) that $L_n$ scales linearly with $n$ when $q<1$. The growth rate of $L_n$ as well the limiting constant has been identified in intermediate regimes in \cite{MueSta11,BhaPel15}. However, so far, scaling limits for $L_n$ have not been established in any case except when $q=1$. In this paper we consider $q\in (0,1)$ fixed. Our main result is a central limit theorem for $L_n(q)$, which is the first result identifying the limiting distribution for the Mallows model in a case when $q \ne 1$. We prove that $L_n$ is asymptotically Gaussian with linear variance, confirming a conjecture in \cite{BhaPel15}. This provides an instance of a phase transition in the scaling limit of the length of LIS between Gaussian and GUE Tracy-Widom distribution as $q$ varies; see Section \ref{s:background} for more details. 

\begin{maintheorem}
\label{t:clt}
Fix $0<q<1$. Then exist constants $\sigma=\sigma(q)>0$ and $a = a(q)>0$ such that for $L_n=L_n(q)$ defined as above, we have
$$ \frac{L_n-an}{\sigma\sqrt{n}} \Rightarrow \mathcal{N}(0,1)$$
as $n\rightarrow \infty$ where $\Rightarrow$ denotes convergence in distribution and $\mathcal{N}(0,1)$ denotes the standard Normal distribution with mean 0 and variance 1.
\end{maintheorem}

We do not have explicit formulae for $a(q)$ and $\sigma(q)$ in the above theorem. However, in Section \ref{s:anscombe}, we derive representations of $a$ and $\sigma$ as certain statistics of a measure on the space of permutations with variable length. Observe that $a=\lim_{n\to \infty} \frac{\E L_n}{n}$, where the existence of the limit is guaranteed by a subadditive argument. Bounds on $a$ were derived in \cite{BhaPel15}, but evaluating this constant was left as one of the open problems there.

Our second main result is a law of large numbers for $L_n^{\downarrow}(q)$ for $0<q<1$. The order of the rate of growth of $L_n^{\downarrow}(q)$ was found in \cite{BhaPel15} and the identification of the limiting constant was left as an open question, which we answer in the following theorem. 

\begin{maintheorem}
\label{t:lln}
Fix $0<q<1$. For $L^{\downarrow}_n=L^{\downarrow}_n(q)$ defined as above, we have
$$\frac{L^{\downarrow}_n\sqrt{\log q^{-1}}}{\sqrt{2 \log n}} \rightarrow 1$$
in probability as $n\rightarrow \infty$.
\end{maintheorem}

The reversal $\pi^R$ of a permutation $\pi$ is given by $\pi^R(i) \defeq n+1 - \pi(i)$ for $1 \le i \le n$. It is easy to check (see e.g. \cite{BhaPel15}) that the reversal of a $\mbox{Mallows}(q)$ distributed permutation is a $\mbox{Mallows}(1/q)$ permutation. Since the length of an LIS of a permutation is equal to the length of an LDS of its reversal, Theorem \ref{t:lln} can also be interpreted as a law of large numbers of $L_n(q)$ when $q>1$.

\subsection{Background and Related Works}
\label{s:background}
The Mallows model was originally introduced motivated by ranking problems in statistics. The Mallows distribution can be defined more generally with respect to a reference permutation $\pi_0$ through the Kendall-Tau distance 
\[
d(\pi,\pi_0) \defeq \sum_{1 \le i < j \le n} \mathbf{1}_{\{\pi_0(i) < \pi_0(j)\}}\mathbf{1}_{\{\pi(i) > \pi(j)\}}
\]
which reduces to ${\rm inv}(\pi)$ when $\pi_0$ is the identity permutation. Mallows was interested in the problem of determining an unknown ``true" ordering or ``reference permutation" on the elements $\pi_0$ given samples of orderings penalized according to the number of pairs out of order compared to the reference permutation. Generalized Mallows models using a variety of metrics on permutations are used to model ranked and partially ranked data in machine learning and social choice theory \cite{Dia88,Cri85,FigVer93,Mar95,DiaRam00,LebRaf02}. 

The ${\rm Mallows}(q)$ measure is related to representations of the Iwahori-Hecke algebra  \cite{DiaRam00} and to a natural $q$-deformation of exchangeability studied by Gnedin and Olshanski \cite{GneOls10, GneOls12}. It also arises in connection with the stationary measures of the biased adjacent transpositions shuffle on $\mathcal S_n$ and of the nearest-neighbor asymmetric exclusion process on an interval where particles jump to the left with probability $1-p$ and to the right with probability $p$ and $q=p/(1-p)$ \cite{BenBerHofMos05}. Further, by the correspondence between the ASEP and the XXZ quantum spin system \cite{Cap05,Sta09}, the ground state of the XXZ model is a projection of the Mallows model on permutations.

The normalizing constant $Z_n(q)$ in the Mallows distribution has a closed form formula which Diaconis and Ram \cite{DiaRam00} observed to be the Poincar\'{e} polynomial
\begin{align}\label{e:poincare}
Z_{n,q} \defeq \sum_{\pi \in \mathcal{S}_n} q^{{\rm inv}(\pi)}= \prod_{i =1}^n \frac{q^i-1}{q-1} = [n]_q! = [n]_q \cdots [1]_q, \mathrm{ \ where \ } [i]_q = \frac{q^i-1}{q-1}.
\end{align}
The formula \eqref{e:poincare} for $Z_{n,q}$ implies a straightforward method for generating a random Mallows distributed permutation. 

%
%

As mentioned before, the question of determining the the length of the longest increasing subsequence of permutations drawn from a Mallows model for general $q$ was raised in \cite{BorDiaFul10}. When $q=1$, i.e.\ in the case of uniform random permutations, the asymptotics of $L_n$ (known as Ulam's problem) have been extensively studied. Vershik and Kerov \cite{VerKer77} and Logan and Shepp \cite{LogShe77} showed that $\E L_n/\sqrt n \to 2$ (see also \cite{AldDia95} for a proof using Hammersley's interacting particle system). Mueller and Starr \cite{MueSta11} first studied $L_n$ under the Mallows measure for $q \ne 1$. In the regime that $n(1-q)$ tends to a constant $\beta$, they established a weak law of large numbers showing that $L_n/\sqrt{n}$ converges in probability to a constant $C(\beta)$ which they determined explicitly. Their arguments rely on a Boltzmann-Gibbs formulation of a continuous version of the Mallows measure and the probabilistic approach of Deuschel and Zeitouni for analyzing the longest increasing subsequence of i.i.d.\ random points in the plane \cite{DeuZei95, DeuZei99}. 

Subsequently, in \cite{BhaPel15}, Bhatnagar and Peled established the leading order behavior of $L_n$ in the regime that $n(1-q) \to \infty$ and $ q \to 1$. They analyzed an insertion process which they called the Mallows process for randomly generating Mallows distributed permutations and showed that $L_n/n\sqrt{1-q}  \to 1$ in $L_p$ for $0<p < \infty$ as $n \to \infty$. They established the order of $L^{\downarrow}_n$ and showed that it grows at different rates for different regimes of $q$ as a function of $n$ (in particular they showed $\E L_n^{\downarrow}(q)=\Theta (\sqrt{\log n/\log q^{-1}})$ when $q\in (0,1)$ is fixed) and proved large deviation bounds for $L_n$ and $L^{\downarrow}_n$. They also established a linear upper bound on the variance of $L_n$ and left open the questions of determining the precise variance and the distribution of $L_n$ for all regimes of $n$ and $q$.

Recently, further progress been made in the analysis of the empirical measure of points corresponding to the Boltzmann-Gibbs measure of Mueller and Starr. Mukherjee \cite{Muk13} determined the large deviation rate function of the empirical measure of points and Starr and Walters recently showed that the large deviation principle has a unique optimizer \cite{StaWal15}.

The case $q=1$ is special because it is one of the exactly solvable models that belong to the so-called KPZ universality class. In this case, the longest increasing subsequence problem can also be represented as a directed last passage percolation problem in a Poissonian environment. The results of \cite{VerKer77,LogShe77} follow from an asymptotic analysis of exact formulae that can be obtained for $\E L_n$ through a combinatorial bijection between $\mathcal S_n$ and Young tableaux known as the Robinson-Schensted-Knuth (RSK) bijection \cite{Rob38,Sch61, Knu70}. The RSK bijection can be further used to obtain the order of fluctuations and scaling limit of $L_n$ in this case. In their breakthrough work, Baik, Deift and Johansson showed that for uniformly random permutations $L_n$ has fluctuations of the order of $n^{1/6}$ and the limiting distribution of $n^{-1/6}(L_n-2\sqrt{n})$ is the GUE Tracy-Widom distribution from random matrix theory \cite{BaiDeiJoh99}. 

When $q\neq 1$, the integrable structure is lost, and the powerful combinatorial, algebraic and analytic tools used in \cite{BaiDeiJoh99} are no longer available for finding the limiting distribution of $L_n$. Indeed, as Theorem \ref{t:clt} shows, for $q$ bounded away from $1$, we get a different  scaling limit, namely Gaussian with diffusive scaling, in contrast  with the Tracy-Widom limit with subdiffusive scaling one gets for $q=1$. 

Theorem \ref{t:clt} indicates that as $q\to 1$, there is a phase transition in the limiting distribution of $L_n$ between Gaussian and Tracy-Widom. There are other models where such transitions have been shown to occur. The so-called BBP transition \cite{BaiBenPec05} for the spiked complex Wishart model, a model for non-null covariance matrices \cite{John01}, is a prime example. Connections with random matrices have been exploited to show similar transitions in exactly solvable models of last passage percolation with external sources \cite{BenCor11, BaiRai01}. Much less is understood in the absence of the exactly solvable machinery. Chatterjee and Dey \cite{ChaDey13} show that the first passage time across thin cylinders obeys a central limit theorem. Dey, Joseph and Peled \cite{DeyJosPel16} consider the similar problem of directed last passage percolation in a Poissonian environment restricted to band of width $n^{\gamma}$ around the diagonal. Relying on finer estimates available for exactly solvable models they establish a sharp transition in $\gamma$ for the limiting law of the length of a maximal path, showing that when $\gamma<2/3$, the limiting distribution is Gaussian. It is known that if $\gamma>2/3$ the limit is again Tracy-Widom. Another recent example of a Gaussian scaling limit in a last passage percolation model is obtained by Houdr\'e and I\c{s}lak \cite{HouIsl14} for the length of the longest common subsequence of random words.

The Mallows model is not known to be exactly solvable and unlike many of the exactly solvable models described above, the location of the transition (or indeed, if there is only one transition or multiple ones) is currently not known. It can be shown (see Remark \ref{r:tw} in Section \ref{s:oc}) that if $q\to 1$ sufficiently fast as $n\to \infty$ that $L_n$, properly centered and scaled, indeed has a Tracy-Widom distribution. To understand this transition(s) is a fascinating problem. By making the calculations in our proof of Theorem \ref{t:clt} quantitative, it may be possible to show that a central limit theorem also holds for $q$ tending to $1$ sufficiently slowly as $n$ goes to infinity. However, in order to avoid technical complications, we do not pursue this direction in this paper.

%
%
%

Our analysis hinges on identifying a regenerative process associated with an infinite Mallows permutation. In \cite{GneOls10}, Gnedin and Olshanski study a notion of $q$-exchangeability in infinite Mallows permutations on $\mathbb Z$.  As in the finite case, there is a natural insertion process for generating infinite Mallows permutations. One key observation in our work is that we can define a certain $\mathbb N$-valued Markov chain associated with this process. The times when the chain reaches $0$ form a set of regeneration times which enables us to view $L_n$ as a sum of longest increasing subsequences restricted to the permutation defined by the interval between the renewal times; see Section \ref{s:mcregen} for details. This view allows us to apply results from renewal theory directly to establish Theorem \ref{t:clt} and makes the analysis rather clean. A more refined analysis of the same Markov chain allows us to establish Theorem \ref{t:lln}. 

One could hope that this or similar constructions might be useful in studying other properties of the Mallows measure. Indeed, shortly before completing this work, we learnt that Gladkich and Peled use a construction of Mallows permutations and an associated Markov chain similar to what we define to analyze the cycle structure of random Mallows distributed permutations \cite{GlaPel16}. For the chain in \cite{GlaPel16}, Gladkich and Peled also analyze the excursions away from zero, and indeed, the two chains share the same return times to zero and they obtain estimates similar to ours.

\subsection*{Organization of the paper} The rest of this paper is organized as follows. In Section \ref{s:mallows-process} we recall the construction of a Mallows permutation via Mallows' process. In Section \ref{s:mcregen} we obtain a regenerative process representation of a Mallows process. In Section \ref{s:finiterenewals} we obtain estimates on the renewal time of the regenerative process by relating it to a return time of a certain Markov chain. Using these estimates we complete proofs of Theorem \ref{t:clt} and Theorem \ref{t:lln} in Section \ref{s:anscombe} and Section \ref{s:lln} respectively. We conclude with some open questions in Section \ref{s:oc}.

\subsection*{Acknowledgements} The authors are grateful to Ron Peled for many helpful conversations and letting us know of the results in \cite{GlaPel16} and \cite{DeyJosPel16}. We thank Allan Sly for useful discussions and Shannon Starr for explaining some of the details of the results in \cite{MueSta11}. We also thank two anonymous referees for an extremely careful reading of the manuscript and many useful comments and suggestions that helped improve the paper. We also thank one of the anonymous referees for pointing out how the expression for the stationary distribution of the Markov chain we used can be simplified (see Remark \ref{rem:Zq}).

\section{Constructing Mallows Permutations}
\label{s:mallows-process}
Fix $0<q<1$. Gnedin and Olshanski \cite{GneOls10,GneOls12} constructed an {\bf infinite  $\mbox{Mallows}(q)$ permutation} on $\mathbb N$ via an insertion process, which we will also refer to as $\mbox{Mallows}(q)$ process, often dropping $q$ from the argument. The process is as follows. Start with an i.i.d.\ sequence $\{Z_i\}_{i\geq 1}$ of $\mbox{Geom}(1-q)$ variables. Construct a permutation $\widetilde{\Pi}$ of the natural numbers inductively as follows: Set $\widetilde{\Pi}(1)=Z_1$. For $i>1$, set $\widetilde{\Pi}(i) = k$ where $k$ is the $Z_i$-th number in the increasing order from the set $\mathbb{N}\setminus \{\widetilde{\Pi}(j): 1\leq j<i\}$.  For example, suppose that the realizations of the first five independent geometrics are $z_1 = 4, z_2 = 1, z_3 = 6, z_4 = 2$, and $z_5 = 3$. Then we have $\widetilde{\Pi}(1) = 4, \widetilde{\Pi}(2) = 1, \widetilde{\Pi}(3) = 8, \widetilde{\Pi}(4) = 3$ and $\widetilde{\Pi}(5) = 6$. We represent the process step-by-step below. Note that in each step, the new element $i$ is placed in the $z_i$-th unassigned position among the currently unassigned positions.
\vspace{0.1in}

\begin{center}

\begin{tabular}{cccccccccc}
$\underline{ \ \ \  } $ & $\underline{ \ \ \  }$ & $\underline{ \ \ \  }$ & 1 & $\underline{ \ \ \  }$ & $\underline{ \ \ \  }$ & $\underline{ \ \ \  }$ & $\underline{ \ \ \  }$ & $\underline{ \ \ \  }$ & $\cdots$  \\
2  & $\underline{ \ \ \  }$ & $\underline{ \ \ \  }$ & 1 & $\underline{ \ \ \  }$ & $\underline{ \ \ \  }$ & $\underline{ \ \ \  }$ & $\underline{ \ \ \  }$ & $\underline{ \ \ \  }$ & $\cdots$  \\
2  & $\underline{ \ \ \  }$ & $\underline{ \ \ \  }$ & 1 & $\underline{ \ \ \  }$ & $\underline{ \ \ \  }$ & $\underline{ \ \ \  }$ & 3 & $\underline{ \ \ \  }$ & $\cdots$  \\
2  & $\underline{ \ \ \  }$ & 4 & 1 & $\underline{ \ \ \  }$ & $\underline{ \ \ \  }$ & $\underline{ \ \ \  }$ & 3 & $\underline{ \ \ \  }$ & $\cdots$  \\
2  & $\underline{ \ \ \  }$ & 4 & 1 & $\underline{ \ \ \  }$ & 5 & $\underline{ \ \ \  }$ & 3 & $\underline{ \ \ \  }$ & $\cdots$  \\
\end{tabular}

\end{center}

Let $\Pi_{n}$ be the permutation on $[n]$ induced by $\widetilde{\Pi}$, i.e., $\Pi_{n}(i)=j$ if  $\widetilde{\Pi}(i)$ has rank $j$ when the set $\{\widetilde{\Pi}(k):k\in [n]\}$ is arranged in an increasing order. Consider in the above example $n=4$. Then we have  ${\Pi}_4(1)=3$, ${\Pi}_4(2)=1$, $\tilde{\Pi}(3)=4$ and $\Pi_{4}(4)=2$. In short, we shall write this permutation as 
$\Pi_4 = 2413$, with the interpretation that $\Pi_4$ takes the element $1$ to the position $3$ and so on. In general, with this representation we can read off $\Pi_{n}$ from the above array by restricting to the elements in $[n]$ in the representation above.  
The following lemma is essentially contained in \cite{GneOls10} although it is not explicitly spelt out there. We provide a proof for completeness.

\begin{lemma}
\label{l:infinitetofinite}
Let $\widetilde{\Pi}$ be an infinite $\mbox{Mallows}(q)$ permutation and let $\Pi_n$ be the induced permutation on $[n]$, as defined above. Then $\Pi_n$ is a $\mbox{Mallows}(q)$ permutation on $[n]$.   
\end{lemma} 

\begin{proof}Fix the values of $\widetilde \Pi (i)$ for $i \in [n]$ and let the set of these values be $A=A_{n}=\{a_i:i\in [n]\}$. Suppose that $\pi \in \mathcal S_n$. Given $A$ and $\pi$, there is a uniquely determined set of values $\{z_i\}_{i=1}^n$ for the first $n$ geometric random variables so that $\Pi_n = \pi$. Moreover,
\[
\P(\Pi_n = \pi , A ) = (1-q)^n \prod_{i=1}^n q^{z_i - 1}.
\]
For any $1 \le i \le n$, let $\pi' = (i,i+1) \circ \pi$, that is, $\pi'$ is the permutation obtained from $\pi$ by exchanging the positions of the elements $i$ and $i+1$. Letting $\{z'_i\}_{i=1}^n$ denote the corresponding geometrics, it is simple to verify that, depending on whether $\pi(i) < \pi(i+1)$ or not, either $z'_i = z_{i+1} + 1$ and $z'_{i+1} = z_i$ or $z'_i = z_{i+1}$ and $z'_{i+1} =z_i -1$ while $z_j = z'_j$ for all $j \notin \{i,i+1\}$. Thus 
\[
\frac{\P(\Pi_n = \pi', A )}{\P(\Pi_n = \pi, A)} = 
\begin{cases}
q & \mathrm{if} \ \pi(i) < \pi(i+1)\\
\frac 1q & \mathrm{if} \ \pi(i) > \pi(i+1)
\end{cases}.
\]
Thus the distribution of $\Pi_n$ conditioned on $A$ is $\mbox{Mallows}(q)$ since the transpositions generate the group $\mathcal S_n$. Summing over all possibilities for $A$ completes the proof of the claim.
\end{proof}

Notice that the construction of $\widetilde{\Pi}$ and $\Pi_n$ as described in Section \ref{s:mallows-process} generates a family of ${\rm Mallows} (q)$ permutations on $[n]$ for $n \in \mathbb N$ on the same probability space.
Henceforth when we talk about a $\mbox{Mallows}(q)$ permutation $\Pi_n$ on $[n]$, it will be assumed that $\Pi_n$ is constructed from an infinite $\mbox{Mallows}(q)$ permutation as described above. 

\section{The Regenerative Process Representation}
\label{s:mcregen}
A stochastic process $\mathbf{X} = \{X(t) \ : \ t \ge 0\}$ is said to be a \textbf{regenerative process} if there exist \textbf{regeneration times} $0 \le T_0 < T_1 < T_2 < \cdots$ such that for each $k  \ge 1$, the process $\{X(T_k+t) \ : \ t \ge 0\}$ has the same distribution as $\{X(T_0+t) \ : \ t \ge 0\}$ and is independent of $\{X(t) \ : \ 0\le t < T_k\}$. Below we define a regenerative process associated with the $\mbox{Mallows}(q)$ process described above.

\begin{figure}[t]
\centering
\includegraphics[width=6cm]{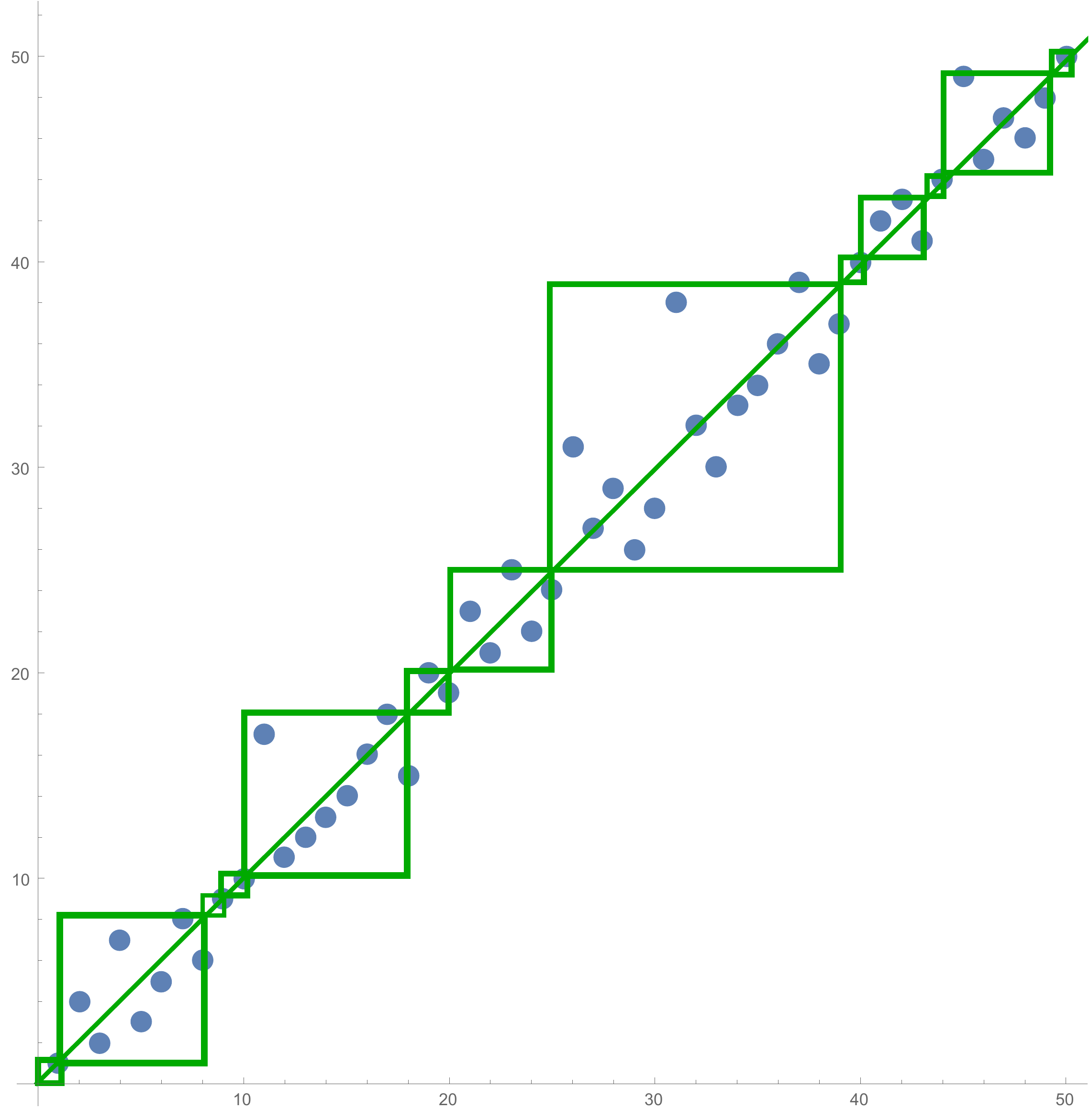}
\caption{The regeneration times $T_0 < T_1 < \cdots$ are marked by the corners of the squares.}
\label{f:renewals}
  \end{figure}
    

Recall the sequential construction of $\widetilde{\Pi}$ and the induced permutation $\Pi_n$. Suppose $m\in \mathbb{N}$ is such that we have $\widetilde{\Pi}([m])=[m]$, i.e.\ the permutation $\widetilde{\Pi}$ restricted to $[m]$ defines a bijection from $[m]$ to $[m]$. Define the permutation $\widetilde{\Pi}^{*}: \mathbb{N}\rightarrow \mathbb{N}$ by $\widetilde{\Pi}^{*}(i)=\widetilde{\Pi}(i+m)-m$. It is clear from the construction of $\widetilde{\Pi}$ that  $\widetilde{\Pi}^{*}$ and  $\widetilde{\Pi}$ have the same law. Together with the independence of the geometric variables $\{Z_i\}$ this implies that $\{\widetilde{\Pi}(i)-i\}_{i\in \mathbb N}$ is a regenerative process with regeneration times $0=T_0<T_1<T_2<\cdots $ where for $i>1$ we have,
$$ T_i=\min\{j>T_{i-1}: \{\widetilde{\Pi}(k): k\in [j]\setminus [T_{i-1}]\}=[j]\setminus [T_{i-1}]\}. $$

We illustrate by an example. Suppose that in the sequential construction for $\widetilde \Pi$, the values of the first 8 geometrics are $z_1 = 1, z_2 = 2, z_3 = 3, z_4 = 1, z_5 = 2, z_6 =3 , z_7 = 1$ and $z_8 = 1$. This corresponds to the permutation
\[1 \ \ 4 \ \ 2 \ \ 7 \ \  3 \ \  5 \ \ 8 \ \ 6 \ \ 9 \ \ 10 \ \ \underline{ \ \ \ } \ \ \underline{ \ \ \ } \ \ \cdots \]
and we see that $T_1 = 1, T_2 = 8, T_3 = 9$ and $T_4 = 10$ are the regeneration times. We may also view the graphical representation of the permutation by plotting the points $(i,\widetilde \Pi(i))$ in $\mathbb R^2$ for each $i \in \mathbb N$. This is illustrated in Figure \ref{f:renewals} for the first 50 elements of $\widetilde \Pi$ which is a random ${\rm Mallows}(0.55)$ permutation. The regeneration times are marked by the corners of the squares which lie on the diagonal $y=x$. The figure illustrates that the points can be partitioned into such squares, which are minimal in the sense that no smaller square with its corners on the diagonal can contain a strict subset of the points in a box.

Set $X_{i}=T_{i}-T_{i-1}$ for $i\geq 1$. Clearly, $X_i$ are independent and identically distributed.  Let $\Sigma_j(i):=\widetilde{\Pi}(i+T_{j-1})-T_{j-1}$ for $i\in \{T_{j-1}+1, T_{j-1}+2, \ldots, T_{j}\}$. Then $\Sigma_j$ is a permutation of $[X_j]$ and furthermore the $\{\Sigma_j\}_{j\geq 1}$ are i.i.d.. Let
$S_n \defeq \min \{ j: T_j\geq n \}$. 

Recall that $L_n$ (resp.\ $L^{\downarrow}_n$) is the length of the longest increasing (resp.\ decreasing) subsequence in $\Pi_n$. The following two lemmas connect $L_n$ and $L^{\downarrow}_n$ with the corresponding statistics defined in the permutations $\{\Sigma_j\}_{j\geq 1}$.


\begin{lemma}
\label{l:approx}
For $j\geq 1$, let $Y_j$ denote length of a longest increasing subsequence of $\Sigma_j$. Then we have,
$$ \sum_{j=1}^{S_n-1} Y_j < L_n \leq \sum_{j=1}^{S_n} Y_j.$$
\end{lemma}

\begin{proof}
By Lemma \ref{l:infinitetofinite}, the LIS of $\Pi_n$ is distributed as $L_n$. 
Observe that any subsequence in $\Sigma_{j}$ corresponds, in an obvious way, to a subsequence in $\widetilde{\Pi}$ and conversely, any subsequence of $\widetilde{\Pi}$ contained in $[T_j]\setminus[T_{j-1}]$ corresponds uniquely to a subsequence in $\Sigma_{j}$.
An increasing subsequence of $\Pi_n$ when restricted to $[T_j]\setminus [T_{j-1}]$ for $1 \le j \le S_n$ corresponds to an increasing subsequence of $\Sigma_j$, which implies the upper bound. On the other hand, any union of increasing subsequences in the $\Sigma_j$ for $1 \le j \le S_n-1$ corresponds to an increasing subsequence in $\Pi_n$, which implies the lower bound. 
\end{proof}

\begin{lemma}
\label{l:ldsapprox}
 Let $Y^{\downarrow}_i$ denote the length of the LDS in $\Sigma_i$. We have 
\begin{align}
\max_{i\leq S_n-1} Y^{\downarrow}_i \leq L^{\downarrow}_n \leq \max_{i\leq S_n} Y^{\downarrow}_i.
\label{e:lds-upper-lower}
\end{align}
\end{lemma}

\begin{proof}
The lemma follows by observing that any decreasing subsequence of $\Pi$ must be contained in $[T_i]\setminus [T_{i-1}]$ for some $i$. We omit the details.  
\end{proof}

\section{Renewal Time Estimates via a Markov Chain}
\label{s:finiterenewals}
Our objective in this section is to prove that the inter-renewal times $X_i$ as defined in the previous section has finite first and second moments. These are the conditions we require to apply results from the theory of regenerative processes to show the central limit theorem for $L_n$. We define a Markov chain such that the $X_i$'s can be represented as the excursion lengths of this Markov chain. Kac's formula for the moments of return times from the theory of recurrent Markov chains then implies that the moments of $X_i$ are finite.

For convenience,  let $X$ denote a random variable with the distribution same as the common one of $X_i$'s.  First we show that $X$ has the same law as the return time of a certain Markov chain which we define below.

Let $\{M_n\}_{n\geq 0}$ denote a Markov chain with the state space $\Omega=\N \cup \{0\}$ and the one step transition defined as follows: set $M_n=\max\{M_{n-1}, Z_n\}-1$ where $\{Z_i\}$ is a sequence of i.i.d.\ $\mbox{Geom}(1-q)$ variables.  Let $R_0^{+}$ denote the first return time to 0 of this chain, i.e.\,
$$R_0^{+}=\min\{k>0: M_k=0\}.$$

\begin{lemma}
\label{l:mc}
For the Markov chain $M_n$ started at $M_0=0$, the return time $R_0^{+}\stackrel{d}{=}T_1$. In particular $X$ has the same law as $R_0^{+}$.
\end{lemma} 

\begin{proof}
Couple the Markov chain $M_n$ with $M_0=0$ with the Mallows' process by using the same sequence $\{Z_i\}$ of random variables to run both processes. Under this coupling, it is easy to verify that for each $n$, by definition
\[
M_n = \max_{1 \le j \le n} \{ \widetilde \Pi(j) \} - n.
\] 
The claim thus follows immediately from the definitions of $R_0^+$ and $T_1$.
\end{proof} 

We analyze the Markov chain $M_n$ and the return time $R_0^{+}$ in the next few lemmas.

\begin{lemma}
\label{l:stationary}
The Markov chain $M_n$ is a positive recurrent  Markov chain whose unique stationary distribution $\mu=(\mu_j)_{j\geq 0}$ is given by 
$$ \mu_j=\biggl( 1+\sum_{j=1}^{\infty} \dfrac{q^{j}}{\prod_{k=1}^{j}(1-q^k)}\biggr) ^{-1} \dfrac{q^{j}}{\prod_{k=1}^{j}(1-q^k)} ;~~ j\geq 0.$$
\end{lemma}

\begin{proof} 
 Let $\pmb{P}=\{P_{i,j}\}_{i,j\geq 0}$ denote the transition matrix of the chain and let $Z$ denote a $\mbox{Geom}(1-q)$ random variable. It is clear from the definition of the chain that for $i\geq 0$ and $j\geq i$ we have  $P_{i,j}=\P(Z=j+1)=q^j(1-q)$; for $i\geq 1$ we have $P_{i,i-1}=\P(Z\leq i)=1-q^i$ and for all other pairs $(i,j)$ we have $P_{i,j}=0$. Clearly the chain is irreducible. It is known from elementary Markov chain theory (see e.g.\ \cite{AldFil99}) that a stationary distribution exists and is unique if and only if there exists a unique probability vector (i.e., vector with non-negative entries whose co-ordinates sum up to 1) $\mu$ solving the set of linear equations $\mu \pmb{P}= \mu$. The equation corresponding to the $j$-th column of the matrix $\pmb{P}$ is given by
\begin{equation}
\label{e:mat0}
\mu_0(1-q)+\mu_1(1-q)=\mu_0 
\end{equation}
for $j=0$ and
\begin{equation}
\label{e:matj}
\sum_{k=0}^{j}\mu_k q^{j}(1-q)+ \mu_{j+1}(1-q^{j+1})=\mu_j 
\end{equation}
for $j>0$.  It is easy to check that any solution of this set of equations must satisfy $\mu_{j+1}=\frac{q}{1-q^{j+1}}\mu_j$, and hence we must have 
$$\mu_{j}= \dfrac{q^{j}}{\prod_{k=1}^{j}(1-q^k)}\mu_0.$$
Since 
$$1+\sum_{j=1}^{\infty} \dfrac{q^{j}}{\prod_{k=1}^{j}(1-q^k)}= \mathcal{Z}(q) < \infty$$
a unique probability vector $\mu$ satisfying the above conditions does indeed exist and is given by 
$$\mu_j = \dfrac{1}{\mathcal{Z}(q)}\dfrac{q^{j}}{\prod_{k=1}^{j}(1-q^k)}.$$
Since the chain is irreducible and has a stationary distribution, it is positive recurrent (see e.g.\ \cite[Theorem 13.4]{AldFil99}).
\end{proof}

\begin{remark}\label{rem:Zq}
In fact, as an anonymous referee has pointed out, the expression for $\mathcal Z(q)$ can be simplified and is given by
\[
\mathcal Z(q) = \frac{1}{\prod_{k=1}^\infty (1-q^k)}.
\]
\end{remark}

The existence of first and second moments of $R_0^{+}$ follows from the above lemma and is proved in Lemma \ref{l:moments}. 
We begin with the following preliminary lemmas. Let $R_i$ denote the time for the chain to reach state $i$. We shall denote by $\E_i$ (resp.\ $\P_i$) the expectation (resp.\ the probability measure) with respect to the chain started at the state $i$ and $\E_\mu$ shall denote the expectation with respect to the chain started at stationarity. 

\begin{lemma}\label{l:dominating}
For all $i \ge 1$, $\E_iR_{i-1}\geq \E_{i+1}R_i$.
\end{lemma}

\begin{proof}
If two copies of the chain are both started at $k \ge i$ and coupled using the same set of geometric variables, so that they are identical, then $R_{i-1} > R_i$ and hence $\E_k R_{i-1} \ge \E_k R_i$. Suppose now that we couple two copies of the chain, one started at $i$ and the other at $i+1$ using the same geometric variables. Let $Z$ be a $\mbox{Geom}(1-q)$ variable and suppose that we make one move according to $Z$ in both chains. If $Z \le i$, then the chain started at $i$ reaches $i-1$ and the chain started at $i+1$ reaches $i$. If $Z \ge i+1 $, then both chains go to the state $Z-1$. Thus,
\[
\E_iR_{i-1} = \P(Z \le i) + \sum_{j=i+1}^\infty \P(Z=j) \E_{j-1} R_{i-1} \ge \P(Z \le i) + \sum_{j=i+1}^\infty \P(Z=j) \E_{j-1} R_i =  \E_{i+1}R_i. \qedhere
\]

\end{proof}

\begin{lemma}
\label{l:mumoment}
For the Markov chain $M_n$, $\E_{\mu}R_0<\infty$.
\end{lemma}

\begin{proof}
\begin{eqnarray*}
\E_{\mu}R_0 &=& \sum_{j=1}^{\infty} \mu_j\E_jR_0\\
&=& \sum_{j=1}^{\infty}\mu_j \sum_{k=1}^{j} \E_{k}R_{k-1}\\
\mathrm{(by \ Lemma \ \ref{l:dominating})} &\leq &  \sum_{j=1}^{\infty}j\mu_j \E_1R_0 <\infty.
\end{eqnarray*}
The last inequality can be justified as follows. The positive recurrence of the chain $M_n$ implies that $E_1R_0$ is finite. Further, Lemma \ref{l:stationary} shows that $\sup_{j} \mu_j/ q^{j}$ is finite (this can be verified by either noting that $\frac{1}{\prod_{k=1}^{\infty}(1-q^{k})} >0$ or by using the formula in Remark \ref{rem:Zq}) and hence $\sum_{j}j \mu_j$ is finite. 
\end{proof}

\begin{lemma}
\label{l:moments}
Let $R_0^+$ be as defined in Lemma \ref{l:mc}. Then we have $\E_0 R_0^{+}< \infty$ and $\E_0 (R_0^{+})^2 < \infty.$
\end{lemma} 

\begin{proof}
It is a basic fact about Markov chains that (see e.g. \cite{AldFil99}) that $\E_0 R_0^{+}= \mu_0^{-1}$. From the proof of Lemma \ref{l:stationary}, we have $\mu_0^{-1} = \mathcal{Z}(q)< \infty$. For the second moment we argue as follows. Let $R_0$ denote the time of the first visit of the Markov chain to the state $\{0\}$.  It is a consequence of Kac's formula (Corollary 2.24, \cite{AldFil99}) that  (see (2.21) in \cite{AldFil99})
\begin{equation}
\label{e:kac}
\E_0 (R_0^{+})^2=\dfrac{2\E_{\mu}R_0+1}{\mu_0}.
\end{equation}
The claim follows by Lemma \ref{l:mumoment} since $\E_{\mu}R_0 < \infty$.
\end{proof}

We shall also need the following tail estimates for the return time to prove Theorem \ref{t:lln}.

\begin{proposition}
\label{l:return}
Let $0<q<1$ and consider the Markov chain $M_n$ as defined above. There exist positive constants $A=A(q)$ and $c=c(q)$ such that for all $t\geq 0$ and $s\geq 0$, we have 
\begin{equation}
\label{e:returntail}
\P_t[R_0^{+}>10t+s]\leq Ae^{-cs}.
\end{equation}
More generally, denoting the first return time to or below $v\geq 0$ by $R_v^{+}$ we have
$$\P_{t+v}[R_v^{+}>10t+s]\leq Ae^{-cs}.$$
\end{proposition}

First we show that it suffices to only prove the first statement in the above proposition. To see this notice the following. If we couple two copies of the chain $M_n$ and $M'_n$ with $M_0=t$ and $M'_0=t+v$ using the same sequence $\{Z_i\}$ of Geometric variables, then we have that $M'_n-M_{n}$ does not increase with $n$ and in particular,  $\min\{n\geq  \ell \geq 1: M_{\ell}\}\geq \min\{n\geq  \ell \geq 1: M'_{\ell}-v\}$. Hence the return time to $0$ in $M_n$ is at least as large as the return time to $v$ in $M'_n$ showing that it is sufficient to establish \eqref{e:returntail}. We shall prove \eqref{e:returntail} using the following estimates.
%



\begin{lemma}
\label{l:lt1}
Fix $0<q<1$. Let $C_1=C_1(q)$ be sufficiently large such that $\frac{q^{C_1}}{1-q}<\frac{1}{10}$. There exist positive constants $A=A(q)>1$ and $c=c(q)$ such that for any $t\geq C_1$ and $s\geq 0$
$$\P_t[R_{C_1}^{+}>10t+s/2] \leq Ae^{-cs}.$$
\end{lemma}

\begin{proof}
Consider the Markov chain $M_n$ with $M_0=t$. Let $\alpha>0$ be a constant that we will choose to be sufficiently small. Observe that 
we have for $qe^{\alpha}<1$,
\begin{eqnarray*}
\E\left(e^{\alpha M_{\ell+1}}\mid M_{\ell}\right) &=& (1-q^{M_{\ell}})e^{\alpha(M_{\ell}-1)}+\frac{q^{M_{\ell}}(1-q)}{1-qe^{\alpha}}e^{\alpha M_{\ell}}\\
&= & e^{\alpha M_{\ell}} \left( e^{-\alpha}(1-q^{M_{\ell}})+ \frac{q^{M_{\ell}}(1-q)}{1-qe^{\alpha}} \right).
\end{eqnarray*}

Since $M_{\ell}>C_1$ on the event $\{\ell < R_{C_1}^{+}\}$ it follows by using this that $\frac{q^{C_1}}{1-q}<\frac{1}{10}$, by choosing $\alpha$ sufficiently small one can make the quantity in the parenthesis above less than $e^{-\alpha/10}$ (for small $\alpha$, it is asymptotically $e^{-9\alpha/10}$). It follows that 
$$\E\left(e^{\alpha M_{10t+s/2}}\mathbf{1}_{\{ R_{C_1}^{+}>10t+s/2\}}\mid M_{0}=t\right)\leq e^{\alpha t}e^{-\alpha(10t+s/2)/10}.$$ 
The lemma now follows.
\end{proof}

Let $\mathcal{L}_x(t)$ denote the time the chain $M_n$ spends at or below $x$ up to time $t$. We have the following estimate.
\begin{lemma}
\label{l:lt2}
Fix $r\leq C_1$ where $C_1$ is as above. Then there exist constants $C_2=C_2(q)>0$ and $c>0$ such that
$$\P_r[\mathcal{L}_{C_1}(s/2)< \frac{s}{C_2(q)}]\leq A_1e^{-cs}.$$
\end{lemma}

\begin{proof}
For two copies of the chain $M_n$ and $M'_{n}$ started at $a$ and $b$ respectively with $a\leq b$, the chains can be coupled so that $M_n \le M'_{n}$, and hence without loss of generality we may assume $r=C_1$. Now let $\xi_1, \xi_2, \ldots$ denote the lengths of a sequence of independent excursions above $C_1$. Hence it suffices to show that for $C_2$ sufficiently large 
$$\P[\sum_{i=1}^{s/C_2}\xi_i > s/2]\leq Ae^{-cs}.$$
This in turn follows by observing that by Lemma \ref{l:lt1} we have $\E e^{\alpha \xi_i} <\infty$ for $\alpha$ sufficiently small.
\end{proof}

Now we are ready to prove Proposition \ref{l:return}.

\begin{proof}[Proof of Proposition \ref{l:return}]

Let $C_1,C_2$ be as in the above two lemmas. It follows from Lemma \ref{l:lt1} that it suffices to prove 
$$\P_{C_1}[R_0^{+}> \frac{s}{2}]\leq Ae^{-cs}$$
for some positive constants $A$ and $c$ for $s$ sufficiently large.
For $i=1,2,\ldots , s/4C_1$, denote the interval $[(2i-2)C_1, (2i-1)C_1)$ (resp.\ $[(2i-1)C_1, 2iC_1)$) by $J_{i}$ (resp.\ $J_i^{*}$). For each $i$, let $Z^{(i)}_{j} Z^{(i,*)}_j, j=1,2,\ldots, C_1$ denote independent sequences of i.i.d.\ $\mbox{Geom}(1-q)$ variables. For the chain $M_n$, define
$$\tau_i= \min\{n\in J_i: M_n\leq C_1\};\quad  \tau^{*}_i= \min\{n\in J^*_i: M_n\leq C_1\}.$$
Let $\mathcal{A}_i$ denote the event 
$$\mathcal{A}_i=\left\{\{i:\tau_{i}<\infty\}< \frac{s}{4C_1C_2}\right\}$$
and define the event $\mathcal{A}_i^{*}$ similarly by replacing $\tau_i$ by $\tau_i^{*}$.
Observe that from Lemma \ref{l:lt2} it follows that 
$$\P_{C_1}[\mathcal{A}_{i}\cap \mathcal{A}_i^{*}]\leq Ae^{-cs}.$$
Let $B$ (resp.\ $B^{*}$) denote the event that $R_0^{+}>\frac{s}{2}$ together with the complement of $\mathcal{A}_i$ (resp.\ the complement of $\mathcal{A}_i^{*}$). Clearly using the above display it suffices to show 
$$\P[B]+\P[B^{*}]\leq Ae^{-cs}$$
for some positive constants $A$ and $c$ for $s$ sufficiently large. Run the chain $M_n$ as follows. Let 
$$\tau_i= \min\{n\in J_i: M_n\leq C_1\}.$$
If $\tau_i<\infty$, then use the variables  $Z^{(i)}_{j}$ to run the chain for the next $C_1$ steps, and use independent external randomness to run the chain for other steps. Call $i$ \textbf{good} if $Z^{(i)}_j=1$ for all $j=1,2,\ldots, C_1$. Clearly if for some $i$, $\tau_i<\infty$ and $i$ is good then $R_0^{+}\leq s/2$. Now observe that $\P[i~\text{is good}]= (1-q)^{C_1}=d(q)>0$ and also observe that on the complement of $\mathcal{A}_i$ we have 
$$\#\{i:\tau_{i}<\infty\}>\frac{s}{4C_1C_2}.$$
It follows that 
$$\P_{C_1}[B]\leq (1-d(q))^{s/4C_1C_2}.$$
Arguing similarly with replacing  $J_i, \tau_i$ and $Z^{(i)}_j$ by $J_i^*, \tau_i^{*}$ and $Z^{(i,*)}_j$ respectively gives us
the same upper bound for $\P_{C_1}[B^*]$. The proposition now follows by noting 
$$\P_{C_1}[R_0^{+}> \frac{s}{2}]\leq \P_{C_1}[B]+P_{C_1}[B^*]+ \P_{C_1}[\mathcal A_i\cap \mathcal A_i^*].$$

\end{proof}

\section{Anscombe's Theorem and a CLT for the length of the LIS}
\label{s:anscombe}
In this section we complete the proof of Theorem \ref{t:clt} by invoking a central limit theorem for a random sum due to Anscombe. Let $X_1,X_2 \cdots$ be i.i.d.\ random variables with finite mean and variance $\sigma^2 >0$ and let $N(t)$ be an integer-valued process defined on the same probability space as the $X_i$. Anscombe's Theorem \cite{Ans52} says that if the partial sums $Q_n$ for the $\{X_i\}$ obey a central limit theorem and do not fluctuate too much, then the random sum $Q_{N(t)}$ also obeys the central limit theorem.

\begin{theorem}[Anscombe's Theorem, e.g. \cite{Gut06}]
\label{t:anscombe}
Let $X,X_1,X_2,\ldots$ be independent, identically distributed random variables with mean $0$ and positive, finite variance $\sigma^2$. For $n \ge 1$, let $Q_n = \sum_{i=1}^n X_i$. Suppose $\{N(t), t\ge 0\}$ is a family of positive, integer values random variables such that for some $0<c < \infty$,
\[
\frac{N(t)}{t} \stackrel{p}{\to} c \mathrm{ \ \ as \ } t \to \infty. 
\] 
Then,
\[
\frac{Q_{N(t)}}{\sqrt{ t}} \stackrel{d}{\to} \mathcal N(0,c\sigma^2) \mathrm{ \ \ as \ } t \to \infty.
\]
\end{theorem}

To apply Anscombe's theorem in our context, we make use of the following concentration result. Recall the regenerative process from Section \ref{s:mcregen} with inter-renewal times $X_i$. Recall $S_n=\min\{j:\sum_{i=1}^j X_i \geq n\}$.

\begin{lemma}\label{l:Snconc}
For $\mu_0$ as defined in the previous section,
\[
\frac{S_n}{n} \stackrel{{\rm a.s.}}{\to} \mu_0.
\]
\end{lemma}

\begin{proof}
Observe that 
$$\frac{ \sum_{j=1}^{S_n-1} X_{j}}{S_n} \leq \frac{n}{S_n}\leq  \frac{ \sum_{j=1}^{S_n} X_{j}}{S_n}.$$
As $n\to \infty$, by strong law, both the left and right hand sides of the above inequality converges to $\mu_0^{-1}$, hence the lemma.
\end{proof}

Using Theorem \ref{t:anscombe} and Lemma \ref{l:Snconc}. we can show the following regenerative version of the Central Limit Theorem (see e.g. \cite[Chapter 2, Theorem 65]{Ser09}), we omit the proof.

\begin{theorem}[Regenerative CLT]
\label{t:regen} 
Let $(X_i,Y_i)_{i\geq 1}$ and $S_n$ be as defined in \S~\ref{s:mcregen}. Define $a:=\mu_0 \E Y_1 < \infty$. Suppose further that $\eta^2 :=\Var (Y_1-aX_1)$ is positive and finite. Set $Q_n=\sum_{j=1}^{S_n} Y_j$. Then we have 
$$ \frac{Q_n-an}{\sqrt{n}} \Rightarrow \mathcal{N}(0, \mu_0\eta^2). $$
\end{theorem}

We need the following to complete the proof of Theorem \ref{t:clt}.

\begin{lemma}
\label{l:positive} In the context of Theorem \ref{t:regen}, $0<\eta^2< \infty$.
\end{lemma}

\begin{proof}
Observe that, since $1\leq Y_1 \leq X_1$, we have  $|Y_1-aX_1|\leq (1+a)X_1$ and hence $\eta^2 <\infty$ using Lemmas \ref{l:mc} and \ref{l:moments}. To see $\eta^2>0$, note that
$$\Var (Y_1-aX_1) = \E ((Y_1-aX_1)^2) \geq \E(\Var (Y_1\mid X_1)) $$
and hence it suffices to prove that for some $j\in \N$ with $\P(X_1=j)>0$, we have $\Var (Y_1\mid X_1=j)>0$. To see this consider $j=3$; it is straightforward to see that $\P(X_1=3)>0$, and conditioned on $\{X_1=3\}$, notice that $\Sigma_1$ can be both the permutations $(3~ 2~ 1)$ and $(3~ 1~ 2)$ with positive probability. It then follows that $\Var (Y_1\mid X_1=3)>0$ and the proof is complete.
\end{proof}

Now we are in a position to complete the proof of Theorem \ref{t:clt}. 
We make use of the following basic result.

\begin{lemma}
\label{l:basic}
Let $W_1, W_2, \ldots,$ be an i.i.d.\ sequence of non-negative random variables with $\E W_i^2< \infty $. Then we have for all constants $C>0$
$$\frac{\max_{1\leq i \leq Cn W_i}}{\sqrt{n}} \rightarrow 0$$
in probability.
\end{lemma}

\begin{proof}
Fix $C>0$. For every $\epsilon>0$ we have
\begin{eqnarray*}
\P\left(\max_{1\leq i \leq Cn} W_i\geq \epsilon \sqrt{n}\right) &=& 1-\left(1-\P\left( \frac{W_1^2}{\epsilon^2}\geq n \right)\right)^{Cn}\\
&\rightarrow & 0
\end{eqnarray*}
as $n \rightarrow \infty$. This follows from the fact that $n\P( W_1^2/\epsilon^2 \geq n ) \rightarrow 0$ as $n\rightarrow \infty$ since $\E [W_1^2/\epsilon^2] <\infty$.
This completes the proof.  
\end{proof}

\begin{proof}[Proof of Theorem \ref{t:clt}]
It follows from Lemma \ref{l:approx} that
$$\frac{Q_n-an}{\sqrt{n}}- \frac{\max_{i\leq S_n} Y_i}{\sqrt{n}}\leq  \frac{L_n-an}{\sqrt{n}} \leq \frac{Q_n-an}{\sqrt{n}}.$$
Note that $\E Y_i^2 <\infty$ since $Y_i \le X_i$ and $\E(X_i^2) < \infty$ by Lemma \ref{l:moments}. Using this and Lemma \ref{l:basic} it follows that
$$\frac{\max_{i\leq 2\mu_0 n} Y_i}{\sqrt{n}} \stackrel{p}{\rightarrow} 0.$$
Using Lemma \ref{l:Snconc} now gives
$$\frac{\max_{i\leq S_n} Y_i}{\sqrt{n}} \stackrel{p}{\rightarrow} 0.$$
Hence setting $\sigma=\mu_0^{1/2}\eta$ and using Theorem \ref{t:regen} we have  
$$ \frac{L_n-an}{\sigma\sqrt{n}} \Rightarrow \mathcal{N}(0,1).$$
This completes the proof.
\end{proof}

\section{Law of large numbers for the length of the LDS}
\label{s:lln}
In this section we establish Theorem \ref{t:lln}, a weak law for the length of the longest decreasing subsequence $L^{\downarrow}_n$ of a ${\rm Mallows}(q)$ permutation, or equivalently, $L_n$ for a ${\rm Mallows} (1/q)$ permutation for $0<q<1$. Our proof makes use of the Markov chain defined in Section \ref{s:mcregen}. Along the way, we show large deviations estimates for the longest decreasing subsequence that improve upon some of the results in \cite[Theorem 1.7]{BhaPel15} and simplify the proofs. \\


Recall the regenerative process from Section \ref{s:mcregen}. Let $\Sigma$ denote a random permutation having the same distribution as $\Sigma_i$. Let  $Y^{\downarrow}$ denote the length of LDS of $\Sigma$. Theorem \ref{t:lln} follows from the following proposition and Lemma \ref{l:ldsapprox}.

\begin{proposition}
\label{p:tail}
$\P(Y^{\downarrow}\geq k)= q^{k^2/2(1+o(1))}$ as $k\rightarrow \infty$.
\end{proposition}

We postpone the proof of Proposition \ref{p:tail} and assuming it, prove the theorem. We use Lemma \ref{l:Snconc}, the fact that $S_n$ is concentrated.

\begin{proof}[Proof of Theorem \ref{t:lln}] 
Fix $\varepsilon>0$. Since the $Y^{\downarrow}_i$ are independent and identically distributed, using Proposition \ref{p:tail} it can be verified that as $n \to \infty$,
\begin{align}
\P\left( \displaystyle\max_{i \le (1-\varepsilon) \mu_0 n } Y^{\downarrow}_i  < (1-2 \varepsilon) \sqrt{\frac{2 \log n}{\log q^{-1}}}\right) \le  \left( 1-\frac{1}{n^{(1+o(1))(1-2\varepsilon)^2)}}\right)^{(1-\varepsilon) \mu_0 n} \to 0
\label{e:lower}
\end{align}
and
\begin{align}
\P\left( \displaystyle\max_{i \le (1+\varepsilon) \mu_0 n } Y^{\downarrow}_i  > (1+2 \varepsilon)\sqrt{\frac{2 \log n}{\log q^{-1}}} \right) \le  (1+\varepsilon) \mu_0 n \frac{1}{n^{(1+o(1))(1+2\varepsilon)^2}} \to 0.
\label{e:upper}
\end{align}

By Lemma \ref{l:Snconc}, with probability going to $1$ as $n \to \infty$, for every $\varepsilon >0$, $(1-\varepsilon)\mu_0 n \le S_n \le (1+\varepsilon) \mu_0 n$. The result thus follows from equations \eqref{e:lds-upper-lower}, \eqref{e:lower} and \eqref{e:upper}.

\end{proof}

We break the proof of Proposition \ref{p:tail} into two parts, proved in the following lemmas.
\begin{lemma}
\label{l:lb}
We have $\P(Y^{\downarrow}=k)\geq  q^{k^2/2(1+o(1))}$ as $k\rightarrow \infty$.
\end{lemma}

\begin{proof}
The probability $\P(Y^{\downarrow} = k)$ can be lower bounded by the probability that $\Sigma$ is the permutation $(k,k-1,\ldots, 2,1)$. Further,
\[
\P(\Sigma {\rm \ is \ the \ permutation \ } (k,k-1,\ldots, 2,1)) = (1-q)^{k} q^{\sum_{i=1}^{k-1} i} = (1-q)^{k} q^{\frac{k(k-1)}{2}} \ge  q^{k^2/2(1+o(1))}. \qedhere
\]
\end{proof}

\begin{lemma}
\label{l:lb1}
We have $\P(Y^{\downarrow}\geq k)\leq  q^{k^2/2(1+o(1))}$ as $k\rightarrow \infty$.
\end{lemma}

To prove Lemma \ref{l:lb1} we need the following lemmas. For the rest of this section, we shall consider the coupling between the Markov chain $M_n$ and the Mallows' process used in construction of $\Sigma$ as described in Lemma \ref{l:mc}.

\begin{lemma}
\label{l:geom}
Suppose $\ell_1<\ell_2<\cdots <\ell_{k}$ are such that $(\ell_k, \ell_{k-1},\ldots , \ell_1)$ is a decreasing subsequence in $\Sigma$. Then $Z_{\ell_1}>Z_{\ell_2}>\cdots >Z_{\ell_k}$. Further, for $i\geq 2$, $M_{\ell_i}=M_{\ell_{i}-1}-1$, and finally, $\displaystyle\min_{\ell_1 \le t < \ell_k} M_t > Z_{\ell_k} -1$.
\end{lemma}

\begin{proof}
For $(\ell_k, \ell_{k-1},\ldots , \ell_1)$ to be a decreasing subsequence we must have that $\ell_{i}$ is placed to the left of $\ell_{i-1}$ for all $1 < i \le k$. By construction $M_{\ell_{i-1}}\leq Z_{\ell_{i-1}}-1$ for all $i$. So at step $\ell_{i-1}$ there are at most $Z_{\ell_{i-1}}-1$ many empty spots to the left of the spot where $\ell_{i-1}$ is placed. So for $\ell_{i}$ to be placed in one of these spots we must have $Z_{\ell_i}<Z_{\ell_{i-1}}$. This proves the first assertion of the lemma.

For the second assertion suppose that for some $i\geq 2$, we have $M_{\ell_i}\geq M_{\ell_i-1}$. Then one must have $Z_{\ell_i}> M_{\ell_i-1}$. This implies that $\ell_i$ is placed to the right of all elements placed so far, in particular to the right of $\ell_{i-1}$, which contradicts the assumption that $(\ell_k, \ell_{k-1},\ldots , \ell_1)$ is a decreasing subsequence. 

For the last assertion, observe that when any $\ell_1 \le t < \ell_k$ is assigned to its position, there must be at least $Z_{\ell_k}$ empty positions to the left of $\ell_1$. This is because it must be the case that $\ell_k$ is assigned to the left of $\ell_1$ since $(\ell_k,\ldots,\ell_1)$ is a decreasing subsequence. Thus, it cannot be the case that $M_t \le Z_{\ell_k } -1$ since $M_t$ counts the total number of unassigned positions to the left of the rightmost assigned position.
%
%
%
\end{proof}

\begin{lemma}
\label{l:censor}
Let $\{Z_i\}_{i\geq 1}$ be a sequence of i.i.d.\ $\mbox{Geom}(1-q)$ random variables. Consider the Markov chain $\{M_t\}_{t\geq 0}$ defined by $M_{t+1}=\max\{M_t,Z_{t+1}\}-1$ started from $M_0=m$. Fix  $0< \ell_1<\ell_2<\cdots < \ell_{k}$ such that $M_{\ell_i}=M_{\ell_{i}-1}-1$ for all $i$. Let $S=\{\ell_1,\ell_2,\ldots, \ell_k\}$. Consider the chain $M'$ started from $m$ which is run using the same sequence of geometric random variables $\{Z_i\}$ except that the $\ell_{i}$-th steps are censored for each $1 \le i \le k$, i.e.\ $M'_{t+1}=\max\{M'_t,Z'_{t+1}\}-1$ where $Z'_i=Z_{f(i)}$ where $f(i)$ is the $i$-th number when $\N\setminus S$ is arranged in increasing order. Then
$$\min _{t\in [\ell_k-k]} M'_t \geq \min_{t\in [\ell_k]} M_t.$$
\end{lemma}

Lemma \ref{l:censor} is an immediate consequence of the following lemma using induction on $k$.

\begin{lemma}
\label{l:censorbasic}
In the set-up of Lemma \ref{l:censor}, suppose $Z_{1}\leq m_1\leq m_2$. Consider running two copies of the chain  $M$ and $M'$ with $M_0=m_1$ and $M'_0=m_2$. Let $M$ evolve using the sequence $\{Z_i\}_{i\geq 1}$ and $M'$ evolve using the sequence $\{Z_i\}_{i\geq 2}$. Then $M'_{t}\geq M_{t+1}$ for all $t$.  
\end{lemma}

\begin{proof}
Since $Z_1\leq m_1$, it follows that $M_1=m_1-1\leq m_2$. 
The result now follows by induction and the definitions of the chains. By definition, $M'_t = \max\{M'_{t-1},Z_{t+1}\} - 1$ and $M_{t+1}  = \max \{M_t, Z_{t+1}\} - 1$. Thus if $M'_{t-1} \ge M_t$, then $M'_{t}\geq M_{t+1}$.
\end{proof}

We are now ready to prove Lemma \ref{l:lb1}.

\begin{proof}[Proof of Lemma \ref{l:lb1}]
Observe that if $Y^{\downarrow}\geq k$, there must exist a sequence $\ell_1<\ell
_2<\cdots < \ell_k$ such that $(\ell_k, \ell_{k-1}, \ldots , \ell_{1})$ is a decreasing subsequence in $\Sigma$ and there does not exist $\ell_0<\ell_1$ such that that $\ell_0$ can be added to the sequence to make a longer decreasing subsequence. Let $\mathbf{l}=\{\ell_1 < \ell_2 <\cdots < \ell_{k}\}$ and $\mathbf{h}=\{h_1 > h_2> \cdots > h_{k}\}$. Let $\mathcal{A}_{\mathbf{l},\mathbf{h}}$ denote the event that $(\ell_k, \ell_{k-1}, \ldots , \ell_{1})$ is a decreasing subsequence of $\Sigma$ satisfying the above property that $\ell_1$ is as small as possible, and $Z_{\ell_i}=h_i$ for all $i$. Clearly 

\begin{equation}
\label{e:sum}
\P[Y^{\downarrow}\geq k] =\sum_{\mathbf{l}, \mathbf{h}} \P[\mathcal{A}_{\mathbf{l},\mathbf{h}}].
\end{equation}

Now, it is easy to observe using Lemma \ref{l:geom} that 
\begin{align}
\label{e:AEFD}
\mathcal{A}_{\mathbf{l},\mathbf{h}}\subseteq \mathcal{E}_{\ell_1}\cap \cf_{\mathbf{h}}\cap \cd_{h_1,h_k,\mathbf{l}}
\end{align}
where 

$$\mathcal{E}_{\ell_1}=\left\{\min_{t\in [\ell_1-1]} M_{t}>0\right\};$$
$$\cf_{\mathbf{h}}=\left\{ \forall i\in [k], \ Z_{\ell_i}=h_i\right\};$$
$$\cd_{h_1,h_k,\mathbf{l}}=\left\{M_{\ell_1}=h_1-1; \forall i\geq 2, \  M_{\ell_i}=M_{\ell_i-1}-1  ; \min_{\ell_1\le t <\ell_k} M_t>h_k-1\right\}.$$

Let $M'_t$ be the chain started at $h_1-1$ so that $M'_0=h_1-1$, $M'_{t}=\max\{Z'_{t},M'_{t-1}\}-1$ and $\{Z'_i\}$ is the sequence of geometrics restricted to $\{Z_i\}_{i=\ell_1+1}^\infty$ omitting the sequence $\{Z_{\ell_i}\}_{i=2}^k$. Let $\cg_{h_1,h_k,\mathbf{l}}$ denote the event that
$$\min_{t\in [\ell_k-\ell_1-k+1]}M'_{t} > h_{k}-1.$$
By Lemma \ref{l:censor},
\begin{align}
\cd_{h_1,h_k,\mathbf{l}} \subseteq \cg_{h_1,h_k,\mathbf{l}}.
\label{e:DinG}
\end{align}

Now observe that $\mathcal{E}_{\ell_1}$, $\cf_{\mathbf{h}}$ and $\cg_{h_1,h_k,\mathbf{l}}$ are independent. Combining equations \eqref{e:AEFD} and \eqref{e:DinG}, we have that

$$\P[\mathcal{A}_{\mathbf{l},\mathbf{h}}]\leq \P_{0}[R_0^{+}>\ell_1-1]\P_{h_1-1}[R_{h_k-1}^{+}>\ell_k-\ell_1-k+1] (q^{-1}(1-q))^{k}q^{\sum_{i} h_i}.$$
Using Proposition \ref{l:return} we have 
$$\P_{0}[R_0^{+}>\ell_1-1]\leq Ae^{-c(\ell_1-1)}$$
and
$$\P_{h_1-1}[R_{h_k-1}^{+}>\ell_k-\ell_1-k+1]\leq Ae^{-c(\max\{\ell_k-\ell_1-k-10(h_1-h_k),0\})}.$$

Observe that $\sum_{i=2}^{k}h_i \geq k(k-1)/2=k^2/2(1+o(1))$. Now we split the sum over $\mathbf{l}$ and $\mathbf{h}$ in the right hand side of \eqref{e:sum} into a few cases. Let $\mathcal{C}_1$ denotes the set of all $\mathbf{l}$ such that $\ell_k\leq k^{3/2}$. Then we have

\begin{eqnarray*}
\sum_{\mathbf{l}\in \mathcal{C}_1, \mathbf{h}} \P[A_{\mathbf{l}, \mathbf{h}}] &\leq & q^{k^2/2(1+o(1))}\sum_{h_{1}} \binom{k^{3/2}}{k}\binom{h_1}{k} q^{h_1}\\
&\leq & q^{k^2/2(1+o(1))}\sum_{h\geq k} h^kq^{h}=q^{k^2/2(1+o(1))}. 
\end{eqnarray*}

Let $\mathcal{C}_2$ denote the set of all $\mathbf{l}$ such that $
\ell_k>k^{3/2}$ and $\ell_1>\ell_k/2$. Then we have 

\begin{eqnarray*}
\sum_{\mathbf{l}\in \mathcal{C}_2, \mathbf{h}} \P[A_{\mathbf{l}, \mathbf{h}}] &\leq & q^{k^2/2(1+o(1))}\sum_{\ell_k\geq k^{3/2}}\sum_{h_{1}} \binom{\ell_k}{k}\binom{h_1}{k} e^{-c\ell_k /3} q^{h_1}\\
&\leq & q^{k^2/2(1+o(1))}\sum_{\ell \geq k^{3/2}} \sum_{h\geq k} \ell ^k e^{-c\ell /3} h^kq^{h}=q^{k^2/2(1+o(1))}. 
\end{eqnarray*}
To aid the reader attempting to verify the calculations, we note that above, as well as in the following estimate, we have not attempted to optimize the constant in the exponent of the bound.

Let $\mathcal{C}_3$ denote all the pairs $(\mathbf{l},\mathbf{h})$ such that $\mathbf{l}\notin \mathcal{C}_1\cup \mathcal{C}_2$ and $h_1<\ell_{k}/200$. Then we have 

\begin{eqnarray*}
\sum_{(\mathbf{l},\mathbf{h}) \in \mathcal{C}_3} \P[A_{\mathbf{l}, \mathbf{h}}] &\leq & q^{k^2/2(1+o(1))}\sum_{\ell_k\geq k^{3/2}}\sum_{h_{1}} \binom{\ell_k}{k}\binom{h_1}{k} q^{h_1} e^{-c\ell_k/10}\\
&\leq & q^{k^2/2(1+o(1))}\sum_{\ell\geq k^{3/2}} \sum_{h\geq k} \ell ^k e^{-c\ell /10} h^kq^{h}=q^{k^2/2(1+o(1))}. 
\end{eqnarray*}

Finally let $\mathcal{C}_4$ denote all the pairs $(\mathbf{l},\mathbf{h})$ such that $\mathbf{l}\notin \mathcal{C}_1\cup \mathcal{C}_2$ and $h_1\geq \ell_{k}/200$. In this case we have 

\begin{eqnarray*}
\sum_{(\mathbf{l},\mathbf{h}) \in \mathcal{C}_4} \P[A_{\mathbf{l}, \mathbf{h}}] &\leq & q^{k^2/2(1+o(1))}\sum_{h_{1}}\sum_{\ell_k\leq 200h_1} \binom{\ell_k}{k}\binom{h_1}{k} q^{h_1}\\
&\leq & q^{k^2/2(1+o(1))} \sum_{h\geq 200k^{3/2}} (200 h)^{k+1} h^kq^{h}=q^{k^2/2(1+o(1))}. 
\end{eqnarray*}

Combining these four cases we complete the proof of the lemma.
\end{proof}

\section{Concluding Remarks and Open questions}
\label{s:oc}
In this paper, based on a regenerative representation of the Mallows process and analysis of an associated Markov chain, we established some limit theorems for the lengths of longest increasing and decreasing subsequences of a $\mbox{Mallows}(q)$ permutation for a fixed $q\in (0,1)$. Many interesting open questions remain. We conclude with a discussion of a few of them.

\begin{enumerate}

\item For which regime of $q$ is the limiting distribution of $L_n$ Tracy-Widom? If $q\to 1$ sufficiently fast as $n\to \infty$ the limiting distribution is Tracy-Widom, but how fast does $(1-q)$ need to decay for this conclusion to hold? Does there exist a range of $q$ where the limiting distribution is neither Gaussian nor Tracy-Widom?

\begin{remark}
\label{r:tw}
Let us parameterize $q = 1-\delta$. For $\delta = o(n^{-2})$ it is straightforward to couple a ${\rm Mallows}(1)$ permutation and a ${\rm Mallows}(1-\delta)$ permutation to agree with high probability so that the total variation distance goes to 0 as $n \to \infty$. Clearly, in this case $L_n$ has the Tracy-Widom distribution when scaled appropriately. Our observation is that it is possible to improve the bound for the regime with Tracy-Widom limit to $\delta = o(n^{-4/3})$. For $q=1-o(n^{-4/3})$, using Lemma 4.2 of \cite{MueSta11}, it is possible to stochastically sandwich $L_{n}(q)$, between the length of LIS of two uniform random permutations of sizes $N_1(n)$ and $N_2(n)$, where $N_1$ and $N_2$ are such that both these when centered by $2\sqrt{n}$ and scaled by $n^{1/6}$ converges weakly to Tracy-Widom distribution.
\end{remark}

\item How does the variance of $L_n$ grow for different rates of $q\to 1$? There is a general linear upper bound on variance available from \cite{BhaPel15}. We expect the variance to go from linear in $n$ to the order of $n^{1/3}$ as $q\to 1$, but it would be interesting to understand the dependence on $q$.

\item Can one prove a law of large numbers for $L_n^{\downarrow}$ for some range of $q$ going to one? It is shown in \cite{BhaPel15} that $\E L_n^{\downarrow}(q)=\Theta (\sqrt{\log n/\log q^{-1}})$ for $q\to 1$ sufficiently slowly, but showing the existence and identification of a limiting constant remains open.

\end{enumerate}

\bibliographystyle{plain}
\bibliography{all}

\end{document}